\documentclass[11pt]{amsart}

\usepackage[margin=1in]{geometry}
\usepackage{amssymb}
\usepackage[hidelinks]{hyperref}
\usepackage[capitalize]{cleveref}
\usepackage{microtype}
\usepackage{comment}

\usepackage{tikz}

\newcommand{\E}{\mathbb{E}}
\def\F{\mathcal{F}}

\def \F{{\mathbb F}}

\crefname{theorem}{Theorem}{Theorems}

\newtheorem{theorem}{Theorem}[section]

\newtheorem{remark}{Remark}[section]
\newtheorem{corollary}[theorem]{Corollary}
\newtheorem{lemma}[theorem]{Lemma}

\newtheorem*{definition*}{Definition}

\theoremstyle{definition}

\title{Generalized point configurations in ${\mathbb F}_q^d$ }

\author[Bright, Fang, Heritage, Iosevich, Jiang, Parshall, and Sun]{Paige Bright, Xinyu Fang, Barrett Heritage, Alex Iosevich, Tingsong Jiang, Hans Parshall, and Maxwell Sun}

\thanks{This work was supported, in part, by the National Science Foundation Grant NSF DMS 2241623 and NSF DMS 1947438. The fourth listed author is supported in part by the National Science Foundation grant NSF DMS 2154232. The authors also wish to thank Williams College, the University of Michigan, and the University of Rochester for their support.}

\begin{document}

\begin{abstract}
In this paper, we generalize \cite{IosevichParshall}, \cite{LongPaths} and \cite{cycles} by allowing the \emph{distance} between two points in a finite field vector space to be defined by a general non-degenerate bilinear form or quadratic form. We prove the same bounds on the sizes of large subsets of $\F_q^d$ for them to contain distance graphs with a given maximal vertex degree, under the more general notion of distance. We also prove the same results for embedding paths, trees and cycles in the general setting.
\end{abstract}

\maketitle

\tableofcontents

\section{Introduction}

The purpose of this paper is to study point configuration problems stemming from Erdos type distance problems in vector spaces over finite fields. Throughout this paper, we let $q$ denote an odd prime power. For $d\geq 2$, $\mathbb{F}_q^d$ is the $d$-dimensional vector space over the field $\mathbb{F}_q$ with $q$ elements. Given $x = (x_1,\dots, x_d) \in \mathbb{F}_q^d$, we let 
\[
\lVert x\rVert = x_1^2 + \dots+ x_d^2.
\]

The main question we seek to address in this paper can be described as follows. How large does a subset of $E\subseteq\F_q^d$ need to be such that $E$ contains a specified graph of points in $\F_q^d$ with distances assigned between edges? In particular, we are interested in embedding \textit{distance graphs} into finite field vector spaces. We call a graph $\mathcal G = (V,E)$ a \textit{distance graph} when for each edge $e\in E$, there is some associated nonzero edge length $\lambda_e\in \F_q^\ast$. Then, we call $X$ an \textit{isometric copy} of $\mathcal G$ when there exists a distance preserving bijection $\varphi: V\to X$ where for each $v,w\in V$ with an edge $e$ connecting $v$ and $w$, we have $|\varphi(v) - \varphi(w)|^2 = \lambda_e$.
Recall from \cite{IosevichParshall} that a \emph{distance graph} $\mathcal{G} = (V,E)$ is defined to be a graph on the vertex set $V$ with edge set $E$, such that to each edge $e \in E$ there is some associated nonzero length $\lambda_{e} \in \mathbf{F}_q^*$. 

Let $\phi:\F_q^d\times\F_q^d\to\F_q$ be a non-degenerate bilinear form or $\phi(x,y)=Q(x-y)$ where $Q$ is a non-degenerate quadratic form. We call this a \emph{distance type function} and say that two points $x,y\in\F_q^d$ have \emph{distance} $\lambda$ if $\phi(x,y)=\lambda$.
We call $X \subseteq \mathbf{F}_q^d$ an \emph{isometric copy} of $\mathcal{G}$ when there exists a distance preserving bijection $\iota : V \rightarrow X$ where for every $v,w \in V$ connected by an edge $e \in E$, $\phi(\iota(v),\iota(w))= \lambda_e$.  

The following theorem was proved by Iosevich and Parshall.

\begin{theorem}[\cite{IosevichParshall}]\label{BoundBasedOnMaxDegree} Let $n, t \in \mathbf{N}$, and let $A \subseteq \mathbf{F}_q^d$ with $|A| \geq 12n^2 q^{\frac{d - 1}{2} + t}$.  Suppose $\phi$ is given by 
\[
\phi(x,y) = ||x-y||.
\]
Then $A$ contains an isometric copy with respect to $\phi$ of every distance graph with $n$ vertices and maximum vertex degree $t$. \end{theorem}

The main analytic input into Theorem \ref{BoundBasedOnMaxDegree} was the special case of Theorem \ref{distancesTheoremCombined} where $\phi$ is the usual distance function. Our main goal in this paper is to prove this theorem stated for a more general $\phi$, which is done in \Cref{sec:functional}. Once we have that, the generalized version of the above theorem follows immediately by the exact same edge deletion induction argument used in \cite{IosevichParshall}. We state our generalization below. The proof is given in \Cref{sec:main}.

\begin{theorem}\label{MainTheorem} Let $n, t \in \mathbf{N}$, and let $A \subseteq \mathbf{F}_q^d$ with $|A| \geq 12n^2 q^{\frac{d - 1}{2} + t}$. 
Then $A$ contains an isometric copy with respect to any distance type function $\phi$ of every distance graph with $n$ vertices and maximum vertex degree $t$. \end{theorem}

We also prove the results for embedding paths, trees and cycles with respect to a distance type function $\phi$ in Sections \ref{sec:path}, \ref{sec:tree} and \ref{sec:cycle} respectively, generalizing those obtained previously in \cite{LongPaths} and \cite{cycles}.

We now set up some notations. 
Let $\chi : \mathbf{F}_q \rightarrow \mathbf{C}$ denote the canonical additive character. 
We find it convenient to use averaging notation for $f : \mathbf{F}_q^d \rightarrow \mathbf{C}$ of
\[
	\E_x f(x) := q^{-d} \sum_{x \in \mathbf{F}_q^d} f(x),
\]
and we will condense multiple averages $\E_{x_1} \E_{x_2} \cdots \E_{x_n}$ as $\E_{x_1, x_2, \ldots, x_n}$.  
Following \cite{IosevichParshall}, we use the normalized $L^2$ norm of $f : \mathbf{F}_q^d \rightarrow \mathbf{C}$ 
\[
	\| f \|_2 := \Big(\E_x |f(x)|^2 \Big)^{1/2}.
\]

\section{Generalized Functional Distance Theorem}
\label{sec:functional}
The following Functional Distance Theorem was previously proven for the standard distance function $||x-y||$ in \cite[Theorem 4]{IosevichParshall} and for the standard inner product $x\cdot y$ in \cite[Theorem 2.1]{covert2008generalized}. In this section, we give the proof for the general version, where we allow the distance type function $\phi$ to be any non-degenerate quadratic form or non-degenerate bilinear form.

Let $\phi:\F_q^d\times\F_q^d\to\F_q$ be a distance type function defined either by a non-degenerate bilinear form or by a non-degenerate quadratic form. More precisely, we have either $\phi(x,y)=B(x,y)$ for some non-degenerate bilinear form $B$ or $\phi(x,y)=Q(x-y)$ for some non-degenerate quadratic form $Q$.
Define
    \[
    d_\lambda(x, y) = \begin{cases}
        q, & \phi(x,y)=\lambda\\
        0, & \text{otherwise}
    \end{cases}.
    \]
Our main distance counting tool is the following.
\begin{theorem}[Generalized Functional Distance Theorem]\label{distancesTheoremCombined} 
For any nonnegative functions $f,g : \mathbf{F}_q^d \rightarrow \mathbf{R}$ and $\lambda \in \mathbf{F}_q^*$,
\[
	\Big|\E_{x,y} f(x)g(y)d_\lambda(x , y) - \E_x f(x) \E_y g(y)\Big| \leq C_\phi q^{-\frac{d-1}{2}} \| f \|_2 \| g \|_2,
\]
where
\[
C_\phi=\begin{cases}
    1, &\text{if $\phi$ is defined by a bilinear form}\\
    2, &\text{if $\phi$ is defined by a quadratic form}
\end{cases}.
\]
\end{theorem}

We prove this result for the bilinear form case and the quadratic form case separately, since the arguments are quite different.

\subsection{Bilinear form case}
In this section, let $\phi:\F_q^d\times\F_q^d\to\F_q$ be a non-degenerate bilinear form. 
The case when $\phi$ is the usual dot product was essentially proven in Theorem 2.1 of \cite{covert2008generalized}. We generalize their argument to an arbitrary non-degenerate bilinear form using the following orthogonality relation.
\begin{lemma}[Orthogonality for a non-degenerate bilinear form]
Then we have 
    \label{lem:orthogonality}
    
    \[
    \sum_{x\in\F_q^d} \chi(\phi(x,y))
    =
    \begin{cases}
        0, & y\neq 0\\
        q^d, & y=0
    \end{cases}.
    \]
    \begin{proof}
        When $y=0$, $\phi(x,y)=0$ for all $x$, so $\chi(\phi(x,y))=1$ for all $x$ and the equality follows. Assume $y\neq 0$. Suppose the bilinear form is given by the symmetric matrix $A=(a_{ij})$ by
        \[
        \phi(x,y) = x^T A y.
        \]
        Then if $x=(x_1,\cdots,x_d), y=(y_1,\cdots,y_d)$, we have
        \[
        \phi(x,y)
        =
        \sum_i x_i \left(\sum_j a_{ij} y_j\right).
        \]
        So
        \[
        \sum_{x\in\F_q^d}\chi(\phi(x,y))
        = \sum_{x_1}\cdots\sum_{x_d} \chi\left(\sum_i x_i \left(\sum_j a_{ij} y_j\right)\right).
        \]
        Since $\phi$ is non-degenerate, there exists an $i$ such that $\sum_j a_{ij} y_j \neq 0$. Without loss of generality, assume $i=d$. Then if we fix $x_1,\cdots,x_{d-1}$ first and sum in $x_d$, we get that the sum vanishes as the input to $\chi$ takes each value in $\F_q$ exactly once. Then summing over all $x_1,\cdots,x_d$  yields the desired result.
    \end{proof}
\end{lemma}

Following the proof of Theorem 2.1 in \cite{covert2008generalized} and applying \Cref{lem:orthogonality} right after the application of Cauchy-Schwarz, one shows that
    \begin{equation}
        \left| \sum_{\phi(x,y) = t}f(x)g(y) - q^{-1}\sum_{x}f(x)\sum_{y}g(y) \right| \le q^{\frac{d-1}{2}} \left( \sum_{x}f(x)^2 \right)^{\frac{1}{2}}\left( \sum_{y}g(y)^2 \right)^{\frac{1}{2}}.
        \label{eq:nondegbilin}
    \end{equation}
The statement in \Cref{distancesTheoremCombined} then follows immediately by re-normalizing.

\begin{remark}
Since we do not have the orthogonality relation for a general degenerate bilinear form, we have to restrict to the non-degenerate case.
\end{remark}

\subsection{Quadratic form case}
The classification theorem of non-degenerate quadratic forms on $\F_q^d$ (see, for example,  \cite[IV.1, Prop. 5]{serre-course-prop5}) states that there are only 2 non-degenerate quadratic forms up to isomorphism on $\mathbf{F}_q^d$, namely, 
\[
X_1^2 + \cdots + X_{n-1}^2 + X_n^2
\]
and
\[
X_1^2 + \cdots + X_{n-1}^2 + aX_n^2
\]
where $a\in\F_q^*$ is a non-square. 


In this section, let $\phi(x,y)=Q(x-y)$ for some non-degenerate quadratic form $Q$ on $\F_q^d$.

\begin{proof}[Proof of \Cref{distancesTheoremCombined} in the case of non-degenerate quadratic forms]
    By the classification theorem, we may write
    \[
    Q(X)=X_1^2 + \cdots + X_{n-1}^2 + aX_n^2
    \]
    for some $a\in\F_q^*$.
Following the exact same argument involving orthogonality and Fourier inversion in the proof of Theorem 4 in \cite{IosevichParshall}, we get
	\[
	\E_{x,y}f(x)g(x + y)d_\lambda(y) - \E_x f(x) \E_y g(y) = \sum_{\xi \in \mathbf{F}_q^d} \widehat{f}(\xi)\widehat{g}(-\xi) \sum_{\ell \in \mathbf{F}_q^*} \chi(-\ell \lambda) \E_{y} \chi(\ell Q(y) + y \cdot \xi).
	\]
	By \cite[Theorem 5.33]{lidl1997finite},
	\[
		\E_{y} \chi(\ell Q(y) + y \cdot \xi) = q^{-d} G(\chi,\eta)^d \chi\Big(-\frac{Q'(\xi)}{4\ell}\Big) \eta(\ell)^{d-1}\eta(a\ell)
	\]
    where $Q'(X)=X_1^2+\cdots+X_{d-1}^2+a^{-1}X_d^2$.
    Recall that we have 
    \[
    |G(\chi,\eta)|=q^{1/2},
    \]
    and that $\eta(a\ell)=\eta(\ell)$ if $a$ is a square and $\eta(a\ell)=-\eta(\ell)$ otherwise, for all $\ell\in\F_q^*$. Therefore we obtain
	\[
		\Big|\E_{x,y}f(x)g(x + y)d_\lambda(y) - \E_x f(x) \E_y g(y)\Big| \leq q^{-d/2} \sum_{\xi \in \mathbf{F}_q^d} |\widehat{f}(\xi)||\widehat{g}(-\xi)| \Big|\sum_{\ell \in \mathbf{F}_q^*} \chi\Big(-\ell \lambda - \frac{-Q'(\xi)}{4\ell}\Big) \eta(\ell)^d\Big|.
	\]
	As $\lambda \neq 0$, we can recognize the sum in $\ell$ as either a Kloosterman sum (when $d$ is even) or a Sali\'{e} sum (when $d$ is odd), with $s=\ell, a=-\lambda$ and $b=\frac{-Q'(\xi)}{4}$, and arrive at
	\[
	\Big|\E_{x,y}f(x)g(x + y)d_\lambda(y) - \E_x f(x) \E_y g(y)\Big| \leq 2q^{-\frac{d - 1}{2}} \sum_{\xi \in \mathbf{F}_q^d} |\widehat{f}(\xi)||\widehat{g}(-\xi)|.
	\]
	Applying Cauchy-Schwarz and Plancherel,
	\[
		\sum_{\xi \in \mathbf{F}_q^d} |\widehat{f}(\xi)||\widehat{g}(-\xi)| \leq \| f \|_2 \| g \|_2,
	\]
	completing the proof.
\end{proof}

\begin{remark}
    The Gauss sum $G(\chi,\eta)$ is defined in \cite[Section 2]{IosevichParshall}, and the Kloosterman sums and Sali\'e sums are defined in \cite[Section 3]{IosevichParshall}. The estimates we used for these sums can also be found in \cite[Section 3]{IosevichParshall}.
\end{remark}


\section{Generalized Distance Graph Embedding Theorems}
\label{sec:main}
We now state our generalization of the theorems on embedding distance graphs with a certain maximal vertex degree. The proofs are exactly analogous to those in \cite{IosevichParshall} using \Cref{distancesTheoremCombined}.

Let $\mathcal{G}$ be a distance graph with vertices labeled as $V = \{v_1, \dots , v_n\}$ and when $v_i$ is adjacent to $v_j$ we label their common edge $e_{i,j}$ and denote its length by $\lambda_{e_{i,j}}\in\F_q^*$. 
For any distance type function $\phi$ as in the previous section, define for any $\lambda\in\F_q^*$,
\[
d_\lambda(x, y) = \begin{cases}
    q, & \phi(x,y)=\lambda\\
    0, & \text{otherwise}
\end{cases}
\]
and
\[
d_{i,j}(x,y) := \begin{cases} d_{\lambda_{e_{i,j}}}(x,y) & \text{if } v_i \text{ and } v_j \text{ are adjacent} \\ 1 & \text{otherwise} \end{cases}.
\]
We denote the count for the number of (ordered) isometric copies of $\mathcal{G}$  with respect to $\phi$ appearing within $A$ as
\begin{equation}\label{edCount}
	\mathcal{N}_{\mathcal{G}}(A) := \E_{x_1, \ldots, x_n} \prod_{i = 1}^n 1_A(x_i) \prod_{j = i + 1}^{n} d_{i,j}(x_i , x_j).
\end{equation}
This is \emph{almost} a normalized count for the number of (ordered) isometric copies of $\mathcal{G}$ appearing within $A$. Now we are ready to state the main theorems. The proofs of the main theorems essentially follow the same arguments in \cite{IosevichParshall}. We include them here for completeness.

\begin{theorem}\label{degreeAsymptoticnew} Let $\mathcal{G} = (V,E)$ be a distance graph with $n$ vertices, $m$ edges, and maximum degree $t$.  Let $A \subseteq \mathbf{F}_q^d$ with $|A| = \alpha q^d$ and
\begin{equation}\label{edAlphaSizenew}
	\alpha \geq 4mq^{t - \frac{d + 1}{2}}.
\end{equation}
Then
\begin{equation}\label{edInductionnew}
	\Big|\mathcal{N}_\mathcal{G}(A) - \alpha^n \Big| \leq 4 m \alpha^{n - 1} q^{t - \frac{d + 1}{2}}.
\end{equation}
\end{theorem}
\begin{proof}
We proceed with induction on the number of edges, $m$. In the case $m=1$, note the maximum degree is $t = 1$ and we may as well assume we only have $n=2$ vertices, since each isolated vertex contributes a factor of $\alpha$ to $\mathcal{N}_\mathcal{G}(A)$. Setting $f=g=1_A$ in \Cref{distancesTheoremCombined} establishes the base case.

We assume that the theorem has been established for graphs with at most $m-1$ edges. It is obvious, but crucial for us, that the maximum degree of subgraphs of $\mathcal{G}$ remain at most $t$ so that we can proceed to apply (\ref{edInductionnew}) to the subgraph of $\mathcal{G}$ formed by removing an edge. Without loss of generality, assume that $v_1$ and $v_2$ are adjacent.
Set
$$
f(x) := 1_A(x)\prod_{j=3}^n d_{1,j}(x,x_j)
$$
and, for $2 \leq i \leq n$,
\[
	g_i(x) := 1_A(x) \prod_{j = i + 1}^{n} d_{i,j}(x , x_j).
\]
Note that $f(x_1)$ and the $g_i(x_i)$ for $2\leq i\leq n$ together account for all edges of $\mathcal{G}$ except for $e_{1,2}$, which we will be deleting.  With this notation, we can rearrange our count \eqref{edCount} as

\begin{equation}\label{rearrangeCountNew}
	\mathcal{N}_{\mathcal{G}}(A) = \E_{x_3, \ldots, x_n} \prod_{i = 3}^n g_i(x_i) \E_{x_1, x_2} f(x_1) g_2(x_2) d_{1,2}(x_1 , x_2).
\end{equation}

Applying \cref{distancesTheoremCombined} to the inner average,
\begin{equation}\label{applyDistancesNew}
\E_{x_1, x_2} f(x_1) g_2(x_2) d_{1,2}(x_1 , x_2) = \E_{x_1} f(x_1) \E_{x_2} g_2(x_2) + \mathcal{E}
\end{equation}
where
\[
	|\mathcal{E}| \leq 2q^\frac{1 - d}{2} \| f \|_2 \| g_2 \|_2.
\]
Since $f$ and $g_2$ amount to normalized indicator functions, we can essentially replace these $L^2$ norms with simple averages provided we are careful about normalization.  Since they each account for at most $t - 1$ edges (recall that $f$ and $g_2$ account for edges incident on $v_1$ and $v_2$ except the edge $e_{1 , 2}$ and that each existing edge contributes a factor of $q$), we have
\begin{align*}
	\| f \|_2^2 &= \E_{x_1} f(x_1)^2\\
	&\leq q^{t - 1} \E_{x_1} f(x_1)
\end{align*}
and similarly $$\| g_2 \|^2_2 \leq q^{t - 1} \E_{x_2} g(x_2).$$  Hence,
\[
	|\mathcal{E}| \leq 2q^{t - \frac{d + 1}{2}} \Big(\E_{x_1} f(x_1)\Big)^{1/2}\Big(\E_{x_2} g(x_2)\Big)^{1/2}
    =2q^{t - \frac{d + 1}{2}}||f||_1^{1/2} ||g_2||_1^{1/2}.
\]
Together with \eqref{rearrangeCountNew} and \eqref{applyDistancesNew}, we see
\[
	\Big|\mathcal{N}_\mathcal{G}(A) - \E_{x_1, \ldots, x_n} f(x_1) \prod_{i = 2}^{n} g_i(x_i) \Big| \leq 2q^{t - \frac{d + 1}{2}} \E_{x_3, \ldots, x_n} \prod_{i = 3}^n g_i(x_i) \| f \|_1^{1/2} \| g_2 \|_1^{1/2}.
\]
This essentially says that, up to an error, the count for copies of $\mathcal{G}$ within $A$ is the same as that for the subgraph of $\mathcal{G}$ formed by deleting the edge $e_{1,2}$.  We invoke our induction hypothesis to observe
\[
	\Big|\E_{x_1, \ldots, x_n} f(x_1) \prod_{i = 2}^{n} g_i(x_i) - \alpha^n\Big| \leq 4(m - 1) \alpha^{n - 1} q^{t - \frac{d + 1}{2}},
\]
in which case we have shown
\begin{equation}\label{edAlmost}
	\Big|\mathcal{N}_\mathcal{G}(A) - \alpha^n\Big| \leq 4(m - 1) \alpha^{n - 1} q^{t - \frac{d + 1}{2}} + 2q^{t - \frac{d + 1}{2}} \E_{x_3, \ldots, x_n} \prod_{i = 3}^n g_i(x_i) \| f \|_1^{1/2} \| g_2 \|_1^{1/2}.
\end{equation}
Now we need only show that we can bring these two error terms in line.  Applying Cauchy-Schwarz,
\[
	\E_{x_3, \ldots, x_n} \prod_{i = 3}^n f_i(x_i) \| f \|_1^{1/2} \| g_2 \|_1^{1/2} \leq \Big( \E_{x_3, \ldots, x_n} \prod_{i = 3}^n g_i(x_i) \Big)^{1/2} \Big(\E_{x_1, \ldots, x_n} f(x_1) \prod_{i = 2}^n g_i(x_i)\Big)^{1/2}.
\]
By our induction hypothesis, and our assumption that $\alpha$ is not too small, we have both
\begin{align*}
	\E_{x_3, \ldots, x_n} \prod_{i = 3}^n g_i(x_i) &\leq 2\alpha^{n - 2}, \text{ and}\\
	\E_{x_1, \ldots, x_n} f(x_1) \prod_{i = 2}^n g_i(x_i) &\leq 2\alpha^n.
\end{align*}
In particular, together with \eqref{edAlmost}, this allows us to conclude \eqref{edInductionnew} as desired.
\end{proof}
Now it suffices to handle degenerate embeddings; in other words, to make sure that in our count $\mathcal{N}_\mathcal{G}(A)$, at least some are made up of distinct vertices. This is done in the following theorem, which immediately implies \Cref{MainTheorem}. 
\begin{theorem}
    Let $n,t \in \mathbf{N}$ and let $A \subseteq \mathbf{F}_q^d$ with $|A| = \alpha q^d$ with $\alpha \geq 12n^2 q^{t - \frac{d + 1}{2}}$.  Then $A$ contains at least $\frac{1}{2}|A|^n q^{-m}$ isometric copies with respect to the distance type function $\phi$ of every distance graph with $n$ vertices, $m$ edges, and maximum vertex degree $t$.
\end{theorem}
\begin{proof}
Let $\mathcal{G}$ be a distance graph with $n$ vertices, $m$ edges, and maximum vertex degree $t$ with the same labeling scheme as before.  We introduce the restricted average
\[
	\E_{x_1, \ldots, x_n}^* := q^{-nd} \sum_{\substack{x_1, \ldots, x_n \in \mathbf{F}_q^d \\ x_1, \ldots, x_n \text{ distinct}}}
\]
and consider
\[
	\mathcal{N}_\mathcal{G}^*(A) := \E_{x_1, \ldots, x_n}^* \prod_{i = 1}^n 1_A(x_i) \prod_{j = i + 1}^{n} d_{i,j}(x_i , x_j),
\]
which only counts genuine isometric copies of $\mathcal{G}$ appearing within $A$.  If we are considering a configuration within $A$ that is detected by $\mathcal{N}_\mathcal{G}$, but not by $\mathcal{N}^*_\mathcal{G}(A)$, then some vertices coincide.  We will show that the contribution when $x_n$ is equal to one of $x_1, \ldots, x_{n - 1}$, the contribution is negligible.  In particular, we will bound
\[
	\mathcal{E}_n := \E_{x_1, \ldots, x_{n - 1}} \prod_{i = 1}^{n - 1} 1_A(x_i) \prod_{j = i + 1}^{n - 1} d_{i,j}(x_i , x_j) q^{-d} \sum_{x_n \in \{x_1, \ldots, x_{n - 1}\}} 1_A(x_n) \prod_{i = 1}^{n - 1} d_{i,n}(x_i , x_n).
\]
Of course, we can express this as
\[
	\mathcal{E}_n = q^{-d} \E_{x_1, \ldots, x_{n - 1}} \prod_{i = 1}^{n - 1} 1_A(x_i) \prod_{j = i + 1}^{n - 1} d_{i,j}(x_i , x_j) \sum_{j = 1}^{n - 1} \prod_{i = 1}^{n - 1}d_{i,n}(x_i , x_j).
\]
When the innermost sum is nonzero, its contribution is at least bounded by $(n - 1)q^t$, since $v_n$ has degree at most $t$.  As we are assuming $\alpha$ is not too small, we can invoke \cref{degreeAsymptoticnew} to conclude
\[
	\mathcal{E}_n \leq 2(n - 1)\alpha^{n - 1} q^{t - d}.
\]
This is because $\mathcal{E}_n = \E_{x_1, \ldots, x_{n - 1}} \prod_{i = 1}^{n - 1} 1_A(x_i) \prod_{j = i + 1}^{n - 1} d_{i,j}(x_i , x_j) = \mathcal{N}_{\mathcal{G'}}(A)$ where $\mathcal{G'}$ is the graph obtained by removing $v_n$ from $\mathcal{G}$.
Applying a symmetric argument for coincidences involving vertices other than $x_n$, we have shown
\[
	|\mathcal{N}_\mathcal{G}(A) - \mathcal{N}^*_\mathcal{G}(A)| \leq 2n^2 \alpha^{n - 1} q^{t - d}.
\]
Combining this with \cref{degreeAsymptoticnew} with our size requirement on $\alpha$ yields
\[
	|\mathcal{N}^*_\mathcal{G}(A) - \alpha^n| \leq 6n^2 \alpha^{n - 1} q^{t - \frac{d + 1}{2}}
\]
which implies that 
$\mathcal{N}^*_\mathcal{G}(A) \geq \frac{1}{2} \alpha^n$.  The result follows by then clearing normalizations.
\end{proof}

\section{Generalized Path Embedding Theorem}
\label{sec:path}
Theorem 1.1 in \cite{LongPaths} and
Theorem 2.2 in \cite{cycles} can also be generalized for general distance functions $\phi$ using \Cref{distancesTheoremCombined}. 
\begin{theorem}
    \label{thm:generalPaths}
    Let $A\subseteq\F_q^d$. Let $\mathcal{P}_k$ denote the number of embeddings of a path of length $k$ in $A$ with respect to the distance type function $\phi$. Then if
    \[
    |A|>\frac{2k}{\ln 2}q^{\frac{d+1}{2}},
    \]
    we have
    \[
    \left|\mathcal{P}_k - \frac{|E|^{k+1}}{q^k}\right|\leq \frac{2k}{\ln 2}q^{\frac{d+1}{2}}\frac{|E|^k}{q^k}.
    \]
\end{theorem}
\begin{proof}
    The argument is identical to those in \cite{LongPaths} and \cite{cycles} by applying \Cref{distancesTheoremCombined}.
\end{proof}

\section{Generalized Tree Embedding Theorem}
\label{sec:tree}
In this section, we generalize the bounds for trees obtained in \cite[Theorem 1.7]{cycles}. The proof follows the same arguments from there and we include them here for completeness. Let $\phi$ be a distance-type function.
For a tree $T$, let $n_T(E)$ be the number of embeddings with respect to $\phi$ of $T$ into $E\subseteq \F_q^d$. 
\begin{theorem}\label{tree counting}
Fix a tree $T$ with $r+1$ vertices and hence $r$ edges.  For $\epsilon>0$, if $|E|>q^{\frac{d+1}{2}+\varepsilon}$ then there is a subset $E^*\subseteq E$ with 
\begin{align*}
|E\setminus E^*|\leq 2q^{-\frac{2\varepsilon}{r+1}}|E|,
\end{align*}
and 
\begin{align*}
\left|n_T(E^*)-\frac{|E|^{r+1}}{q^r}\right|\leq 8\frac{|E|^{r+1}}{q^{r}}q^{-\frac{2\epsilon}{r+1}}
\end{align*}
\end{theorem}
\begin{proof}
Recall for $t\in\F_q^*$, we defined
\[
d_t(x, y) = \begin{cases}
    q, & \phi(x,y)=t\\
    0, & \text{otherwise}
\end{cases}.
\]
For $\lambda=\lambda(q)>0$ to be specified later (with the property that $\lambda(q)\to\infty$), let
\begin{align*}
E^*&=\left\{x\in E: \ \sum_y{E(y)d_t(x,y)}\leq \lambda \frac{|E|}{q}\right\}.
\end{align*}

Then $|E\setminus E^*|\leq 2\frac{|E|}{\lambda}$ for $q$ sufficiently large, since
\begin{align*}
|E\setminus E^*|\leq \frac{q}{\lambda|E|}\sum_{x,y}{E(x)E(y)d_t(x,y)}\leq 2\frac{|E|}{\lambda}.
\end{align*}
Note that our theorem follows by setting $\lambda=q^{\frac{2\varepsilon}{r+1}}$ in the following lemma.
\begin{lemma}\label{tree lemma}
If $T$ has $r+1$ vertices and hence $r$ edges, then 
\begin{align*}
\left|n_T(E^*)-\frac{|E|^{r+1}}{q^r}\right|
\leq 4r\left(\lambda^{-1}+\lambda^{\frac{r-1}{2}}\frac{q^{\frac{d+1}{2}}}{|E|}\right).
\end{align*}
\end{lemma}
\begin{proof}[Proof of Lemma \ref{tree lemma}]
We proceed by induction on $r$, noting for the case $r=1$ that 
\begin{align*}
\left|n_T(E^*)-\frac{|E^*|^2}{q}\right|\leq 2q^{\frac{d-1}{2}}|E^*|
\end{align*}
Moreover $|E|-|E^*|\leq 2\frac{|E|}{\lambda}$, so 
\begin{align*}
\left|n_T(E^*)-\frac{|E|^2}{q}\right|&\leq \left|n_T(E^*)-\frac{|E^*|^2}{q}\right|+\frac{1}{q}\left||E|^2-|E^*|^2\right| \\
&\leq 2q^{\frac{d-1}{2}}|E|+\frac{4|E|^2}{q\lambda}
\end{align*}
Fix a degree 1 vertex $v$ of $T$ with unique neighbor $w$, and let $T'$ be the tree obtained from $T$ by deleting $v$, and the edge from $v$ to $w$.  We will assume the result holds for $T'$, and deduce that it holds for $T$.  For $x\in E^*$, let $T'(x)$ be the number of embeddings $\sigma$ of $T'$ into $E^*$ with $\sigma(w)=x$.  
Then we have by \Cref{distancesTheoremCombined}, after renormalization,
\begin{align*}
n_T(E^*)&=\sum_{x,y\in \mathbb{F}_q^d}{E^*(y)d_t(x,y)T'(x)}
=q^{-1}||T'||_{1} |E^*|+R, \ \ \text{where} \nonumber \\
R&\leq 2q^{\frac{d-1}{2}}||T'||_{2}|E^*|^{1/2}.
\end{align*}
For the purpose of this proof, we use a different normalization for the $L^1$ and $L^2$ norms of a function from $\F_q^d$ to $\mathbb{C}$. Namely
\[ {||f||}_{2}=\left(\sum_{x \in {\Bbb F}_q^d} {|f(x)|}^2\right)^{1/2},\]
and 
\[ {||f||}_1=\sum_{x \in {\Bbb F}_q^d} {|f(x)|}.\]
By our inductive assumption, 
\begin{align*}
\left|||T'||_1-\frac{|E|^r}{q^{r-1}}\right|
\leq (4r-4)\frac{|E|^{r}}{q^{r-1}}\left(\lambda^{-1}+\lambda^{\frac{r-2}{2}}\frac{q^{\frac{d+1}{2}}}{|E|}\right).
\end{align*}
and in particular for $q$ sufficiently large we have $||T'||_1\leq 2\frac{|E|^r}{q^{r-1}}$.  Thus,
\begin{align*}
\left|q^{-1}||T'||_1|E^*|-\frac{|E|^{r+1}}{q^r}\right|
&\leq q^{-1}||T'||_1(|E|-|E^*|)+\left|q^{-1}||T'||_1|E|-\frac{|E|^{r+1}}{q^r}\right| \\
&\leq \frac{|E|^{r+1}}{q^r}\left((4r-2)\lambda^{-1}+(4r-4)\lambda^{\frac{r-2}{2}}\frac{q^{\frac{d+1}{2}}}{|E|}\right)
\end{align*}
So to prove the lemma, it only remains to bound $R$, i.e. to bound $||T'||_{2}$.  Let
\begin{align*}
D=\max_{x\in E^*}\sum_y{E^*(y)d_t(x,y)},
\end{align*}
then 
\[
D \leq \lambda\frac{|E|}{q}.
\]
Also, trivially for any $x\in E^*$, $T'(x)\leq D^{r-1}$.  Therefore,
\begin{align*}
||T'||_{2}^2\leq \max_x{T'(x)}\cdot ||T'(x)||_{1}
\leq 2\left(\frac{\lambda|E|}{q}\right)^{r-1}\frac{|E|^r}{q^{r-1}}
=2\lambda^{r-1}\frac{|E|^{2r-1}}{q^{2r-2}}
\end{align*}
and thus
\begin{align*}
R\leq 4\lambda^{\frac{r-1}{2}}q^{\frac{d+1}{2}}\frac{|E|^{r}}{q^r},
\end{align*}
We conclude that
\begin{align*}
\left|n_T(E^*)-\frac{|E|^{r+1}}{q^r}\right|
\leq 4r\frac{|E|^{r+1}}{q^r}\left(\lambda^{-1}+\lambda^{\frac{r-1}{2}}\frac{q^{\frac{d+1}{2}}}{|E|}\right)
\end{align*}
completing the proof of the lemma.
\end{proof}
\end{proof}

\section{Generalized Cycle Embedding Theorem}
\label{sec:cycle}
In this section, we generalize the results for embedding cycles obtained in \cite{cycles}. Let $C_n$ denote the number of embeddings of an $n$-cycle into $E$, and let $C_n^*$ be the number of non-degenerate ones.

\begin{theorem}\label{cyclesMain}
Let
\begin{align*}
\gamma=\left\{
\begin{array}{ll}
-1 & :d=2 \\
-\frac{d-2}{2} & :d\geq 3
\end{array}
\right.
\end{align*}
If
\begin{align*}
12q^{\gamma}+8\frac{q^{d+2}}{|E|^2}+\left(24+12\left\lfloor \frac{n}{2}\right\rfloor\right)\frac{q^{\frac{d+1}{2}}}{|E|}\leq 1,
\end{align*}
then for $n\geq 6$,
\begin{align*}
\left|C_n-\frac{|E|^n}{q^n}\right|
\leq \frac{|E|^n}{q^n}\left(12q^{\gamma}+8\frac{q^{d+1}}{|E|^2}+\left(24+12\left\lfloor \frac{n}{2}\right\rfloor\right)\frac{q^{\frac{d+1}{2}}}{|E|}\right)
\end{align*}
Moreover,
\begin{align*}
\left|C_4-\frac{|E|^4}{q^4}\right|
&\leq \frac{|E|^4}{q^4}\left(12q^{\gamma}+8\frac{q^{d+2}}{|E|^2}+28\frac{q^{\frac{d+1}{2}}}{|E|}\right) \\
\left|C_5-\frac{|E|^5}{q^5}\right|
&\leq \frac{|E|^5}{q^5}\left(12q^{\gamma}+8\frac{q^{\frac{2d+3}{2}}}{|E|^2}+32\frac{q^{\frac{d+1}{2}}}{|E|}\right).
\end{align*}
\end{theorem}
This theorem says that if $|E|$ is much bigger than $q^{\frac{d+2}{2}}$ for $q$ sufficiently large, then $C_n$ is very close to $\frac{|E|^n}{q^n}$.  For the case $n=4$, we cannot get a nontrivial result without the restriction on the size of $E$.  However, for larger $n$ we can accept a weaker restriction on the size of $E$, at the cost of a weaker error term.  The techniques for proving these theorems are essentially the same.
\begin{theorem}\label{cyclesMain2}
For $n\geq 5$ and $q$ sufficiently large,
\begin{align*}
\left|C_n-\frac{|E|^n}{q^n}\right|
\leq \left(36+80\cdot 6^{\left \lfloor \frac{n}{2}\right\rfloor-2}+12\left\lfloor \frac{n}{2}\right\rfloor\right)\frac{|E|^n}{q^n}q^{-\left(\frac{n}{2}-1\right)\delta}
\end{align*}
whenever
\begin{align*}
|E|\geq \left\{\begin{array}{ll}
q^{\frac{1}{2}\left(d+2-\frac{k-2}{k-1}+\delta\right)} & : n=2k, \ \text{even} \\
q^{\frac{1}{2}\left(d+2-\frac{2k-3}{2k-1}+\delta\right)} & :n=2k+1 \ \text{odd}
\end{array}
\right.
\end{align*}
where 
\begin{align*}
0<\delta<\frac{1}{2\left\lfloor\frac{n}{2}\right\rfloor^2}
\end{align*}
\end{theorem}
Our final variant of Theorem \ref{cyclesMain} is to count non-degenerate cycles.
\begin{theorem}\label{nondegenerateCycles}
For $n\geq 4$ and $q$ sufficiently large, if
\begin{align*}
|E|\geq \left\{\begin{array}{ll}
q^{\frac{1}{2}\left(d+2-\frac{k-2}{k-1}+\delta\right)} & :n=2k \\
q^{\frac{1}{2}\left(d+2-\frac{2k-3}{2k-1}+\delta\right)} & :n=2k+1
\end{array}
\right.
\end{align*}
then
\begin{align*}
\left|C_n^*-\frac{|E|^n}{q^n}\right|\leq \frac{|E|^n}{q^n}\left(K_nq^{-\left(\frac{n}{2}-1\right)\delta}+2nq^{-\frac{2}{n-1}}+c_nq^{-\frac{d-3}{2}-\varepsilon}\right)
\end{align*}
where $K_4=48$, and 
\begin{align*}
K_n&=36+80\cdot 6^{\left\lfloor\frac{n}{2}\right\rfloor-2}+12\left\lfloor\frac{n}{2}\right\rfloor \ \ \text{for} \ n\geq 5 \\
c_n&=(n-1)^{n-3}\cdot 2^{\binom{n-1}{2}-n+3} \\
\varepsilon&=\left\{\begin{array}{ll}
1-\frac{k-2}{k-1}+\delta & : n=2k \\
1-\frac{2k-3}{2k-1}+\delta & :n=2k+1
\end{array}
\right.
\end{align*}
\end{theorem}

To prove \Cref{cyclesMain} and \Cref{cyclesMain2}, it suffices to prove the following Two-Edge Functional Distance Theorem, which is analogous to \cite[Theorem 3.1 \& 3.2]{cycles}. The remaining arguments are exactly follow those in the original paper. To further prove the non-degenerate counts in \Cref{nondegenerateCycles}, simply make use of the results on trees proved in the previous section and apply the same degeneracy handling strategy as in Section 4 of \cite{cycles}.

\begin{theorem}\label{thm:two_edge_functional}
For nonnegative functions $f,g:\mathbb{F}_q^d\times\mathbb{F}_q^d\to \mathbb{R}$, let $F(x)=\sum_y{f(x,y)}$, $G(z)=\sum_w{g(z,w)}$, $F'(y)=\sum_x{f(x,y)}$, $G'(w)=\sum_z{g(z,w)}$.  Then for nonzero $t\in \mathbb{F}_q$,
\begin{align*}
&\left|\sum_{\phi(x,z) = \phi(w,y) = t}{f(x,y)g(z,w)}-q^{-2}||f||_{1}||g||_{1}\right| \\
& \ \ \ \ \ \ \leq 3q^{-\frac{d+2}{2}}||f||_1||g||_1 + 4q^{d-1}||f||_{2}||g||_{2}+
4q^{\frac{d-3}{2}}\left(||F||_{2}||G||_{2}+||F'||_{2}||G'||_{2}\right)
\end{align*}
In the case $d=2$, one has
\begin{align*}
&\left|\sum_{||x-z||=||w-y||=t}{f(x,y)g(z,w)}-q^{-2}||f||_{1}||g||_{1}\right| \\
& \ \ \ \ \ \ \leq 3q^{-3}||f||_1||g||_1 + 4q||f||_{2}||g||_{2}+
4q^{-\frac{1}{2}}\left(||F||_{2}||G||_{2}+||F'||_{2}||G'||_{2}\right)
\end{align*}
\end{theorem}

\subsection{Proof of \Cref{thm:two_edge_functional}}
The proof when $\phi$ is a non-degenerate bilinear form follows exactly from the proof of \cite[Theorem 3.2]{cycles}  after replacing dot product there by $\phi$, since all that was used about the dot product was its bilinearity and the orthogonality relation, which we proved in \Cref{lem:orthogonality}.

Now assume that $\phi(x,y)=Q(x-y)$ where $Q$ is a nondegenerate quadratic form on $\F_q^d$. We may assume, without loss of generality, that
\[
Q(x) = x_1^2 + x_2^2 + \cdots + ax_d^2
\]
where $a\in\F_q^*$.
Let
\begin{equation}
    S_t = \{x \in \mathbb{F}_q^d : Q(x) = t\}.
\end{equation}
Then, the following is a direct consequence of \cite[Theorem 7.1]{GroupActions}.

\begin{lemma}\label{sphere bounds nonzero}
When $t\neq 0$,
\begin{align*}
|S_t| &= q^{d-1}+\mathcal{E}, \ \text{with} \\
|\mathcal{E}|&\leq q^{\frac{d}{2}}.
\end{align*}
Moreover, in the case $d=2$,
\begin{align*}
|S_t|&=q\pm 1
\end{align*}
\end{lemma}

We now show the following, which is shown in \cite{IR07} in the case $Q = \| \cdot \|$.

\begin{lemma}
 If $t\neq 0$ and $m\neq 0$, then
\begin{align*}
|\hat{S}_t(m)|\leq 2q^{-\frac{d+1}{2}}
\end{align*}
\end{lemma}
\begin{proof}
Let $Q'(x) = x_1^2 + x_2^2 + \cdots + a^{-1}x_d^2$. Then,
\begin{align}
    \hat{S}_t(m) &= q^{-d}\sum_{x \in \mathbb{F}_q^d: Q(x) = t} \chi(x \cdot m) \nonumber\\
    &= q^{-d-1}\sum_{x \in \mathbb{F}_q^d} \sum_{j \in \mathbb{F}_q} \chi(j (Q(x) - t))\chi(x \cdot m) \nonumber\\
    &= q^{-d-1}\sum_{j \in \mathbb{F}_q^*} \chi(-jt) \sum_{x \in \mathbb{F}_q^d} \chi(j Q(x) + x \cdot m) \nonumber\\
    &= q^{-d-1}\sum_{j \in \mathbb{F}_q^*} \chi(-jt) G(\chi,\eta)^d \chi\Big(-\frac{Q'(m)}{4j}\Big) \eta(j)^{d-1}\eta(aj) \label{quadratic_form_gauss_sum} \\
    &= q^{-\frac{d}{2}-1} \left| \sum_{j \in \mathbb{F}_q^*} \chi\Big(-jt - \frac{Q'(m)}{4j}\Big)\eta(j)^{d} \right| \label{quadratic_form_guass_sum_estimate} \\
    &= 2q^{-\frac{d+1}{2}}. \label{quadratic_form_kloosterman_salie_sum}
\end{align}
We have that \eqref{quadratic_form_gauss_sum} follows from \cite[Theorem 5.33]{lidl1997finite}, \eqref{quadratic_form_guass_sum_estimate} follows from the Gauss sum estimate $|G(\chi, \eta)| = q^{1/2}$, and \eqref{quadratic_form_kloosterman_salie_sum} follows from the Sali\'{e} and Kloosterman sum bounds as in the proof of \Cref{distancesTheoremCombined}.
\end{proof}

\begin{corollary}
\begin{align*}
|S_t|\leq 2q^{d-1}, \ \text{and}
\end{align*}
\begin{align*}
|S_t|^2&=q^{2d-2}+\mathcal{E}', \text{with} \\
|\mathcal{E}'|& = |2q^{d-1}\mathcal{E}+\mathcal{E}^2| \leq 3q^{\frac{3d-2}{2}}
\end{align*}
In the case $d=2$, 
\begin{align*}
|\mathcal{E}'|\leq 3q
\end{align*}
\end{corollary}

Using the results above, we can proceed as in the proofs of \cite[Theorem 3.1 \& Lemma 3.7]{cycles} to conclude the proof of \Cref{thm:two_edge_functional}.

\providecommand{\bysame}{\leavevmode\hbox to3em{\hrulefill}\thinspace}
\providecommand{\MR}{\relax\ifhmode\unskip\space\fi MR }
\providecommand{\MRhref}[2]{%
  \href{http://www.ams.org/mathscinet-getitem?mr=#1}{#2}
}
\providecommand{\href}[2]{#2}

\end{document}